\documentclass[12pt]{amsart}

\usepackage{graphicx}
\usepackage{fncylab,hyperref,environ}

\theoremstyle{plain}
\newtheorem{thm}{Theorem}

\newtheorem{lem}{Lemma}

\newtheorem{cor}{Corollary}

\newtheorem*{cnj}{Conjecture}
\theoremstyle{definition}

\newcommand{\idd}{\operatorname{Id}}
\newcommand{\cnv}{\operatorname{conv}}
\newcommand{\crd}{\operatorname{cr}}

\newcommand{\trc}{\operatorname{trace}}
\newcommand{\vll}{\operatorname{vol}}

\newcommand{\eps}{\varepsilon}

\makeatletter

\newcounter{asyfigcntr}
\@ifundefined{asy}{%
  \NewEnviron{asy}{%
    \par\stepcounter{asyfigcntr}%
    \includegraphics{\jobname-\theasyfigcntr}%
    \par%
  }
}{}

\begin{document}

\title{When is the ball a local pessimum for covering?}

\author{Yoav Kallus}
\address{Yoav Kallus, Center for Theoretical Science, Princeton University, Princeton, New Jersey 08544}

\date{\today}

\begin{abstract}
We consider the problem of identifying the worst point-symmetric shape for covering $n$-dimensional
Euclidean space with lattice translates. Here we focus on the dimensions where
the thinnest lattice covering with balls is known and ask whether the ball is
a pessimum for covering in these dimensions compared to all point-symmetric convex shapes.
We find that the ball is a local pessimum in 3 dimensions,
but not so for 4 and 5 dimensions.
\end{abstract}

\maketitle

\section{Introduction}

The lattice sphere packing and covering problems can be stated in similar ways: in both
problems we look for an optimal arrangement of equal-sized spheres centered
at points of a lattice; whereas in the packing problem
we must have no multiply-covered volume and we must minimize the amount of uncovered volume, in the
covering problem we must have no uncovered volume and we must minimize the amount of multiply covered volume.
The lattice sphere packing problem has attained great importance partly because many of the lattices
that give good packing fractions in low dimensions are related to objects of exceptional symmetry in geometry \cite{SPLAG,thompson}.
By contrast, the lattices which give good covering fractions in low dimensions do not seem to be imbued with
symmetry to the same extent. In fact, many highly symmetric lattices seem instead to be locally pessimal at
covering with spheres \cite{vallentin}. Instead of lattices that are pessimal
at covering with spheres relative to nearby lattices, in this article we are interested in
shapes that are pessimal for covering relative to nearby shapes.

In two and three dimensions it has been shown that the rounded octagon and the ball
respectively are locally pessimal for packing among convex point-symmetric shapes\cite{kalluspack,Nazarov},
and in both cases these shapes have been conjectured to be globally pessimal \cite{Gardner95,Reinhardt}.
Reinhardt's conjecture about the rounded octagon has been an open problem since 1934.
In the case of covering, it has been shown by L. Fejes T{\'o}th that the disk is the pessimal
shape for covering the plane \cite{toth-cover}. In this article we show that in three dimensions, the ball
is locally pessimal for covering, but is not so in four and five dimensions. We limit our attention to these
dimensions because those are the only dimensions where the lattice sphere covering problem is solved.
In all these dimensions, $A_n^*$ is known to be the optimal sphere covering lattice \cite{SPLAG}.
Nevertheless, we establish results that would make it easy to determine whether the $n$-dimensional
ball is locally pessimal for covering if the optimal sphere covering lattice in that dimension were to become known.

The investigation in this article follows similar lines to the one in Ref. \cite{kalluspack}. While
many of the concepts and the results in the case of covering are analogous to those in the case
of packing and require only minor changes, others require gross changes. With this in mind, we
have set out to write this article so as to be self-contained. To this end, much of the
content of this paper is closely similar (sometimes nearly verbatim) to the content of Ref.\ \cite{kalluspack}.
We omit a result and cite Ref. \cite{kalluspack} only when the result
is nearly identical, not merely analogous. Many of the concepts have identical names to the
analogous concepts in the case of packing, except with the addition of the modifier ``covering''.
For convenience we usually drop this modifier, implicitly referring throughout the paper out of two
analogous concepts to the one that deals with covering.

\section{Convex Bodies and Lattices}

An $n$-dimensional {\it convex body} is a convex, compact subset of $\mathbb{R}^n$
with a nonempty interior. A body is {\it symmetric about a point} (or point-symmetric) if 
there is a point $\mathbf x$ such that $2\mathbf x-K=K$, and the body is \textit{origin-symmetric}
if that point is the origin. In this article we discuss only such bodies, and
we will implicitly assume that every body mentioned is symmetric about a point.
We denote by $B^n$ the Euclidean unit ball of $\mathbb{R}^n$.
The space of point-symmetric convex bodies $\mathcal K^n_0$ in $\mathbb{R}^n$
is a metric space equipped with the Hausdorff metric
$\delta_H(K,K')=\min\{\eps : K\subseteq K'+\epsilon B^n, K'\subseteq K+\epsilon B^n\}$.
The set of bodies $K$ satisfying $a B^n\subseteq K\subseteq b B^n$ for $b>a>0$ is compact \cite{Gruber}.

Let $S^{n-1} = \partial B^n$ be the unit sphere.
The {\it radial function} of a body in the direction $\mathbf{x}\in S^{n-1}$
is given by $r_K(\mathbf{x}) = \max\{\lambda:\lambda\mathbf{x}\in K\}$.
A body is uniquely determined by its radial function since $K=\bigcup_{\mathbf{x}\in S^{n-1}}[0,r_K(\mathbf{x})]\mathbf{x}$.
For origin-symmetric bodies, the radial distance is an even function.

An $n$-dimensional {\it lattice} is the image of the integer lattice $\mathbb{Z}^n\subseteq\mathbb{R}^n$
under some non-singular linear map $T$. The determinant $d(\Lambda)$ of a lattice $\Lambda=T\mathbb{Z}^n$
is the volume of the image of the unit cube under $T$ and is given by $d(\Lambda)=|\det T|$.
The space $\mathcal L^n$ of $n$-dimensional lattices can be equipped with the
metric $\delta(\Lambda,\Lambda')=\min\{||T-T'|| : \Lambda = T\mathbb{Z}^n, \Lambda'= T'\mathbb{Z}^n\}$,
where $||\cdot||$ is the Hilbert-Schmidt norm.

We call $\Lambda$ a covering lattice for $K$ if for each point $\mathbf{x}\in\mathbb{R}^n$, there
is a lattice point $\mathbf{l}\in\Lambda$ such that $\mathbf{x}\in K+\mathbf{l}$, i.e.,
$\{K+\mathbf{l} : \mathbf{l}\in\Lambda\}$ is a covering of $\mathbb{R}^n$.
The density of this covering is given by $\vll K / d(\Lambda)$, and must be
greater than or equal to one. The set of lattices $\Lambda$, such that $\Lambda$ is
a covering lattice for $K$ and $d(\Lambda)\ge a$ for some body $K$ and $a>0$,
is a compact set \cite{Gruber}.

The {\it critical (covering) determinant} $d_K$ is the maximum, necessarily attained due
to compactness, of all determinants of covering lattices for $K$. 
A lattice attaining this maximum is called a {\it critical (covering) lattice} of $K$. If a covering lattice
of $K$ locally maximizes the determinant amongst covering lattices of $K$,
it is called an {\it extreme (covering) lattice} of $K$.
Clearly, if $K'\supseteq K$, then
$d_{K'}\ge d_{K}$. If this inequality is strict whenever $K'$ is a proper superset of $K$,
we say that $K$ is an {\it inextensible} body \cite{kallusn2}. The optimal covering
fraction for $K$ is $\vartheta(K) = \vll K / d_K$. Note that $\vartheta(TK)=\vartheta(K)$ for any
nonsingular linear transformation $T$. Therefore, we may define $\vartheta$ as a function over the
space of linear classes of $n$-dimensional bodies, equipped with the Banach-Mazur distance
$\delta_{BM}([K],[L])=\min\{t:L'\subseteq K'\subseteq e^t L',K'\in[K],L'\in[L]\}$.
Since this space is compact, there must be a body $K$ with the highest possible
optimal covering fraction amongst all $n$-dimensional bodies. We call this a
{\it (globally) pessimal} body for covering. If a body attains the highest possible
optimal covering fraction in a neighborhood of itself with respect to Hausdorff
distance, then we say it is a \textit{local pessimum} for covering.
A locally pessimal body is necessarily inextensible, but the converse
is not necessarily true.

Below we show that the unit ball is locally pessimal for $n=3$, 
and extensible for $n=4$ and $5$.

\section{Primitive simplices and semieutaxy}

The Voronoi polytope $P_\mathbf{l}$ of a lattice point $\mathbf{l}\in\Lambda$ is the set of all points for which $\mathbf{l}$ is the closest
lattice point, that is, $P_\mathbf{l}=\{\mathbf{x}\in\mathbb{R}^n : ||\mathbf{x}-\mathbf{l}||\le||\mathbf{x}-\mathbf{l'}||\text{ for all }
\mathbf{l'}\in\Lambda\}$. Note that $P_\mathbf{l}=P_0+\mathbf{l}$. The Voronoi polytopes of
the lattice points of $\Lambda$ form the cells of a $\Lambda$-periodic honeycomb, which we call the Voronoi honeycomb of $\Lambda$.
If the combinatorial type of the Voronoi polytope $P_0$
(equivalently, the combinatorial type of the Voronoi honeycomb) is the
same as for all lattices in a neighborhood of $\Lambda$, we say that $\Lambda$ is simple.
If $\Lambda$ is simple, then each vertex of the Voronoi honeycomb lies at the intersection
of $n+1$ cells \cite{barnesdickson}. Therefore, modulo translations by vectors of $\Lambda$, each vertex of the Voronoi polytope $P_0$
is equivalent to exactly $n$ others and all equivalent vertices are equidistant from the origin. Therefore,
the Voronoi polytope can be described as the convex hull of simplices, each with a circumscribing sphere
centered at the origin and with vertex separations in $\Lambda$. We call these simplices the primitive simplices of $\Lambda$.
Whenever $S$ is a primitive simplex of $\Lambda$, so is $-S$.

A Delone simplex of $\Lambda$ whose circumcenter is at some vertex $\mathbf{x}$ of $P_0$,
when translated by $-\mathbf{x}$, is simply $-S$, where $S$ is
the unique primitive simplex with vertex $\mathbf{x}$. For a simple lattice, the Delone triangulation retains its
combinatorics for nearby lattices $T\Lambda$, where $||T-\idd||$ is small enough, and the
primitive simplices of $T\Lambda$ are simply translates of the images under $T$ of the primitive
simplices of $\Lambda$.

A lattice $\Lambda$ is a covering lattice
for the ball of radius $r$ if and only if $r\ge \mu(\Lambda)=\max_S\crd(S)$, where
the maximum runs over all primitive simplices $S$ of $\Lambda$, and $\crd(S)$ denotes
the circumradius of $S$. We call $\mu(\Lambda)$ the covering radius of $\Lambda$,
and we refer to the primitive simplices of $\Lambda$ attaining $\mu(\Lambda)$ as the maximal primitive simplices.
We denote the set of maximal primitive simplices of $\Lambda$ by $X(\Lambda)$.
Since a Delone simplex with exterior circumcenter must have an adjacent Delone
simplex with a larger radius, the circumcenters of maximal primitive simplices lie
in the simplex, and in the interior for simple lattices.

\begin{lem}\label{trlem}Let $\Lambda$ be a simple lattice of covering radius $1$.
There are $C>0$ and $\eps_0>0$ such that whenever
$K$ is a nearly spherical body in the sense that $(1-\eps)B^n\subseteq K\subseteq(1+\eps)B^n$
and $\eps<\eps_0$, then $\Lambda$ is a covering lattice of $K$ if and only if
there exist $\mathbf{t}_S$ for all $S\in X(\Lambda)$ such that $S+\mathbf{t}_S\subseteq K$
and $||\mathbf{t}_S||<C\eps$.\end{lem}

\begin{proof}First let us assume that $\Lambda$ is a covering lattice of $K$. Note that for each
simplex $S\in X(\Lambda)$, $-S$ is a translate of a Delone simplex, and that the bodies $K+\mathbf{l}$,
where $\mathbf{l}$ runs over the vertices of the Delone simplex, must cover the Delone simplex.
As a simple consequence of the nerve theorem (Theorem 10.7 of Ref.\ \cite{nerve}),
there must be a point $\mathbf{x}$ common to all $n+1$ bodies, and that point
is at a distance smaller than $C\eps$ from the simplex circumcenter, since it lies in the intersection
of the balls of radii $1+\eps$ centered at the vertices.
Since $\mathbf{x}\in K+\mathbf{l}$for all vertices $\mathbf{l}$ of the Delone simplex, the points
$\mathbf{x}-\mathbf{l}$ are in $K$, and their convex hull, which is a translate of $S$ is contained in $K$.

Now let us assume that $S+\mathbf{t}_S\subset K$ for all $S\in X(\Lambda)$. By the fact that
$K$ is nearly spherical we also have that $K$ contains all the non-maximal primitive simplices
of $\Lambda$. Since the primitive simplices have interior circumcenters, the
$S+\mathbf{t}_S\subset (1+\eps)B^n$ implies that $||\mathbf{t}_S||<C\eps$.
We will show that $\Lambda$ is a covering lattice for $P'$ the convex hull of $S+\mathbf{t}_S$ where
$S$ runs over all primitive simplices of $\Lambda$ and $\mathbf{t}_S=0$ for non-maximal $S$.

Consider the Voronoi polytope $P_0$ and form a triangulation of it,
i.e., a subdivision of $P_0$ into simplices with no new vertices and
such that any two simplices intersect at a common face or not at all.
When repeated for all $P_\mathbf{l}$ this is a $\Lambda$-periodic
triangulation of $\mathbb{R}^n$. Now, leaving the combinatorics
of the triangulation unchanged, let us translate each
of its vertices by $\mathbf{t}_S$ whenever the
vertex is equivalent to a vertex of $S$ modulo translation
by vectors of $\Lambda$. If $\eps$ is small enough, the result
is still a triangulation of $\mathbb{R}^n$.
The cells obtained by the union of the simplices whose union
previously gave the cells of the Voronoi honeycomb also form a $\Lambda$-periodic
subdivision of space. These cells are in general not convex,
but their convex hulls are lattice translates of $P'$.
Therefore, the lattice translates of $P'$ cover $\mathbb{R}^n$ and $\Lambda$
is in fact a covering lattice for $P'$ and for $K\supseteq P'$.
\end{proof}

Let $S$ be a primitive simplex of a simple lattice $\Lambda$ 
and let $\mathbf{x}_1,\ldots,\mathbf{x}_{n+1}$ be its vertices.
We define a symmetric linear map associated with $S$ as follows:
$$Q_S(\cdot) = \sum_{j=1}^{n+1} \alpha_j \langle\mathbf{x}_j,\cdot\rangle\mathbf{x}_j$$
where $\alpha_j$ are the barycentric coordinates of the circumcenter of $S$:
\begin{equation}\label{alpheqn}
    \sum_{j=1}^{n+1} \alpha_j \mathbf{x}_j =0\text{ and } \sum_{j=1}^{n+1}\alpha_j=1\text.
\end{equation}
Note that $Q_S=Q_{-S}$. The importance of $Q_S$ can
be seen for example in the following lemma. Here we work with
the linear space $\mathrm{Sym}^n$ of symmetric linear maps $\mathbb{R}^n\to\mathbb{R}^n$
equipped with the inner product $\langle Q,Q'\rangle = \trc QQ'$.

\begin{lem}\label{crlem}
Let $S$ be a simplex inscribed in the unit sphere centered at the origin and
let $T$ be a nonsingular linear map. Then
$$\crd(TS)^2 = 1 + \langle M,Q_S\rangle + O(||M||^2)\text,$$
where $M=T^T T-\idd$ and the error term is non-negative.
\end{lem}

\begin{proof}
Let $\mathbf{x}_i$, $i=1,\ldots,n+1$, be the vertices of $S$.
The center $\mathbf{y}$ and radius $R=\sqrt{1+a}$ of the circumsphere of $TS$
are determined by the $n+1$ equations $||T(\mathbf{x}_i)-\mathbf{y}||^2 = 1 + a$,
$i=1,\ldots,n+1$. Defining the $n+1$-element vector $\mathbf{y}'$ whose first
$n$ elements give $\mathbf{y}$ and its last elements is $\tfrac{1}{2}(a-||\mathbf{y}||^2)$,
we can write the system of equations as a linear one:
$2A(T\oplus 1)\mathbf{y}'=\mathbf{b}$ where
$$A = \left(\begin{array}{ccccc}
x_{11}&x_{12}&\cdots&x_{1n}&1\\
x_{21}&x_{22}&\cdots&x_{2n}&1\\
\vdots&\vdots&\ddots&\vdots&\vdots\\
x_{(n+1)1}&x_{(n+1)2}&\cdots&x_{(n+1)n}&1\end{array}\right)\text{,}$$
$$b_i = \left(||T(\mathbf{x}_i)||^2-1\right) = \langle \mathbf{x}_i, M\mathbf{x}_i\rangle\text,$$
and $T\oplus 1$ is the direct sum of $T$ with the $1\times1$ unit matrix.
We can easily recover the circumcenter and radius by inverting
the linear system: $\mathbf{y}' = (1/2)(T\oplus1)^{-1}A^{-1}\mathbf{b}$.
Clearly then,
$$a = (a-||\mathbf{y}||^2) + ||\mathbf{y}||^2 =
\langle\mathbf{c},\mathbf{b}\rangle+O(\eps^2)\text,$$
where $\mathbf{c}$ is the bottom row of $A^{-1}$.
By definition, the elements of $\mathbf{c}$ satisfy
$\sum_{i=1}^{n+1}c_i x_{ij}=0$ for $j=1,\ldots,n$
and $\sum_{i=1}^{n+1}c_i=1$, so these are the same
coefficients $\alpha_i$ of \eqref{alpheqn}. In summary, we have that
$$\begin{aligned}R^2 &= 1+a = 1 + \sum_{i=1}^{n+1}\alpha_i \langle \mathbf{x}_i, M\mathbf{x}_i\rangle + O(||M||^2)\\
&= 1 + \langle M,Q_S\rangle+O(||M||^2)\text,\end{aligned}$$
and the error term, given by $||\mathbf{y}||^2$, is non-negative.
\end{proof}

Let $X$ be a symmetric set of simplices inscribed in $S^{n-1}$ (that is,
the simplex $S$ is in $X$ if and only if $-S$ is too). If there are non-negative
numbers $\upsilon_S$, $S\in X$, such that
\begin{equation}\label{semieut}
    \upsilon_S = \upsilon_{-S} \text{ and } \idd = \sum_{S\in X} \upsilon_S Q_S\text,
\end{equation}
then we say that $X$ is semieutactic.
We say that a set of simplices $S$ is redundantly semieutactic
if the set $X\setminus\{S,-S\}$ is semieutactic for all $S\in X$.

In dimensions $n=2,3,4,$ and $5$, the critical covering lattice for $B^n$ is known,
and it is unique up to rotation. In all of these cases, the lattice is $A_n^*$, the
reciprocal of the root lattice $A_n$ \cite{SPLAG}.
The lattice $A_n^*$ can be embedded isometrically as a rescaled $n$-dimensional
sublattice of the $(n+1)$-dimensional integer lattice $\mathbb{Z}^{n+1}$ as follows:
\begin{equation}
    A_n^* \simeq \lbrace a^{1/2} \mathbf{z} : \mathbf{z}\in\mathbb{Z}^{n+1}, z_i-z_j\in (n+1) \mathbb{Z}, \sum_{i=1}^{n+1}z_i = 0\rbrace\text,
\end{equation}
where $a = 2/{\binom{n+2}{3}}$. The Voronoi polytope of $A_n^*$, also known as
the permutohedron, is the convex hull of the point $\mathbf{x}_0=a^{1/2} (-n/2,-n/2+1,\ldots,n/2)$
and all $(n+1)!$ points derived from it by coordinate permutations. Every vertex
of the permutohedron is related by lattice translations to exactly $n$ other
vertices, and so $A_n^*$ is
a simple lattice and there are $n!$ primitive simplices, all maximal, which divide into
$n!/2$ origin-symmetric pairs. One of the primitive simplices $S_0$ is given by
the $n+1$ cyclic permutations of $\mathbf{x}_0$, and the matrix giving the map $Q_{S_0}$
is a circulant matrix. The symmetric maps associated
with the other simplices are given by simultaneous permutations of the rows and columns of $Q_{S_0}$.
Note that the maps are described by $(n+1)\times(n+1)$ matrices, but they preserve
the linear space spanned by our lattice in $\mathbb{R}^{n+1}$ and act as the zero map
on the orthogonal complement. Therefore, the identity
map we are looking for is not given by the identity matrix, but by the circulant
matrix whose first row is $(\tfrac{n}{n+1},-\tfrac{1}{n+1},\ldots,-\tfrac{1}{n+1})$.

Since by symmetry $\sum_{S\in X(A_n^*)}\tfrac{1}{(n-1)!}Q_S$ must be circulant with
all off-diagonal elements equal, and since its trace is $n$, it must be the identity map.
Therefore, $\mathbf{u}_0 = (1/(n-1)!)_{S\in X(A_n^*)}$ is a solution to
to \eqref{semieut}, and so $X(A_n^*)$ is semieutactic
for all $n\ge2$. Note that the linear system \eqref{semieut} has
more equalities than variables for $n\le3$ and more variables than equalities for $n\ge4$.
Therefore, for $n\ge4$ the solution space must contain a line $\mathbf{u}_0 + t \mathbf{u}_1$,
and this line must contain a non-negative solution where at least one coefficient
$\upsilon_{S}$ vanishes. Therefore, $X\setminus\{S,-S\}$ is semieutactic and
$X$ is not redundantly semieutactic. On the other hand, it is easy to check
that for $n=2$ and $n=3$, the solution to \eqref{semieut} is unique.

We can now prove three results relating these eutaxy
properties with the existence or nonexistence of
certain covering lattices for certain bodies. The first
is a sufficient condition (and necessary under the assumption of simplicity), originally proved by
Barnes and Dickson, for a lattice to be extreme for $B^n$.

\begin{thm}\label{extthm}\emph{(Barnes and Dickson \cite{barnesdickson})}
Let $\Lambda$ be a simple lattice such that the circumradius
of its maximal primitive simplices is $1$. The following are equivalent:
\begin{enumerate}
\item $\Lambda$ is extreme for $B^n$;
\item $X(\Lambda)$ is semieutactic.
\end{enumerate}
\end{thm}

\begin{proof}
Suppose first that $X(\Lambda)$ is not semieutactic. By the fundamental theorem of linear algebra
and the fact that a subspace of $\mathbb{R}^m$ does not contain a non-zero non-negative vector if and only if
its orthogonal complement contains a positive vector, we conclude that there exists a symmetric map $M$
such that $\langle M,Q_S\rangle<0$ for all $S\in X(\Lambda)$ and $\trc M>0$. Let $T_\eps=\sqrt{\idd+\eps M}$,
where the square root $\sqrt{A}$ of a positive definite symmetric map $A$ is meant to denote the unique
positive definite map $B$ such that $B^2=A$. From Lemma \ref{crlem}, as long as $\eps$ is small enough, $\crd(T_\eps S)<1$
for all $S\in X(\Lambda)$ and $S+\mathbf{t}_S\subseteq B^n$.
Therefore $\Lambda$ is a covering lattice of $T_\eps^{-1} B^n$.
Equivalently, $T_\eps\Lambda$ is a covering lattice of $B^n$. Also, 
$\trc M>0$ implies that $\det T_\eps>1$ for small enough $\eps$, so $\Lambda$ is not extreme.

Conversely, suppose that $\Lambda$ is not extreme. Then for arbitrarily
small $\eps$ there exists a map $T$ satisfying $||T-\idd||<\eps$, $\det T>1$,
and $cr(S)\le1$ for all primitive simplices $S$ of $T\Lambda$.
Since the primitive simplices of $T\Lambda$ are just the images under $T$ of the primitive
simplices of $\Lambda$, we have from Lemma \ref{crlem} that $\langle M,Q_S\rangle\le0$
for all $S\in X(\Lambda)$, where $M=T^T T-\idd$. Moreover, since $\det T>1$, we have that $\trc M>0$.
Note that the map $M' = M - \dfrac{\trc M}{2n}\idd$ satisfies $\trc M'>0$ and $\langle M,Q_S\rangle<0$
for all $S\in X(\Lambda)$. Again, from the fundamental theorem of linear algebra,
the existence of such a map $M'$ implies that $X(\Lambda)$ is not semieutactic.
\end{proof}

Moreover, by compactness
we have that for all $\delta>0$, there exist $\eps>0$, such that whenever
$d(\Lambda)>d_{B^n}-\eps$ and $\Lambda$ is a covering lattice for $B^n$, then there is a lattice $\Lambda_0$, critical
for $B^n$, such that $\delta(\Lambda,\Lambda_0)<\delta$. 
Also, whenever $\Lambda$ is a critical lattice for a nearly spherical body $K$ satisfying
$(1-\eps)B^n\subset K\subset(1+\eps)B^n$, then there is a lattice $\Lambda_0$, critical
for $B^n$, such that $\delta(\Lambda,\Lambda_0)<\delta$. 
In cases where the critical covering lattice for $B^n$ is unique up to rotations and simple,
we can prove the following necessary and sufficient condition for $B^n$ to be extensible.

\begin{thm}Let $\Lambda_0$ be the unique (up to rotation) critical covering lattice of $B^n$,
and let $\Lambda_0$ be simple. Then the following are equivalent:
\begin{enumerate}
\item $B^n$ is extensible;
\item $X(\Lambda_0)$ is redundantly semieutactic.
\end{enumerate}
\end{thm}

\begin{proof}First let us assume that $X(\Lambda_0)$ is not redundantly semieutactic.
That is, we assume that there is a maximal primitive simplex $S_0\in X(\Lambda_0)$
such that $X(\Lambda_0)\setminus \{\pm S_0\}$ is not semieutactic.
Consider the $\eps$-symmetrically augmented ball
$B_\varepsilon=\cnv (B^n,\pm(1+\eps)\mathbf{p})$,
where $\mathbf{p}\in S^{n-1}$. 
We are free to assume that $\Lambda_0$ is rotated so that $\mathbf{p}$ is one of the vertices of $S_0$.
By exactly the same argument as in the proof of Theorem \ref{extthm}, we conclude that
there exists a linear map $T$ such that $\det T>1$, $\crd(TS)<1$ for all $S\in X(\Lambda_0)\setminus\{\pm S_0\}$,
and $||T-\idd||$ is arbitrarily small. In fact, if $||T-\idd||$ is small enough,
then $TS_0$ can be translated so as to lie within $B_\eps$ and therefore $T\Lambda_0$ is
a covering lattice of $B_\eps$ by Lemma \ref{trlem}.
Since $d(T\Lambda_0)>d(\Lambda_0)$, and since for each proper $K\supset B^n$, there exists $\varepsilon>0$ such that
$K\supset B_\varepsilon\supset B^n$, it follows that $B^n$ is inextensible.

Consider now a critical lattice $\Lambda$ of $B_\eps$, and note that for
any $\delta>0$, we may choose $\eps>0$ and $\Lambda_0$, a critical lattice
of $B^n$, such that $\delta(\Lambda,\Lambda_0)<\delta$.
We assume that $X(\Lambda_0)$ is redundantly semieutactic. 
Since $\Lambda = T\Lambda_0$ is
a covering lattice of $B_\varepsilon$, then by Lemma \ref{trlem},
$TS+mathbf{t}_S\subseteq B_\varepsilon$ for all $S\in X(\Lambda_0)$, and if $\eps$
is small enough then $TS+mathbf{t}_S\subseteq  B^n$ for all but one pair $S=\pm S_0$.
Since $\Lambda_0$ is redundantly semieutactic, the requirement that $\crd(TS)\le1$
whenever $S\in X(\Lambda_0)\setminus\{\pm S_0\}$ necessarily
implies, when $||T-\idd||$ is small enough, that $\det T\le1$.
Of course, since $d_{B_\varepsilon}\ge d_{B^n}$, we must have
$d_{B_\varepsilon}= d_{B^n}$, and $B^n$ is extensible.
\end{proof}

\begin{cor}$B^n$ is extensible for $n=4$ and $5$, and inextensible for $n=2$ and $3$.\end{cor}

\begin{proof}
Recall that $A_n^*$ is the unique critical lattice for $B^n$ in dimensions $n=2,3,4,$
and $5$ up to rotations. Recall also that $X(A_2^*)$ and $X(A_3^*)$ have a unique, positive
solution to the eutaxy equation \eqref{semieut}, and are therefore not redundantly semieutactic,
and that $X(A_4^*)$ and $X(A_5^*)$ are redundantly semieutactic.
\end{proof}

We now focus on the case where $B^n$ is inextensible.
Particularly, we will assume that the critical lattice $\Lambda_0$ is
unique, simple, and has a unique set of eutaxy coefficients satisfying
\eqref{semieut}. This is the case for $n=2$ and $3$.

\begin{thm}\label{detthm}Let $\Lambda_0$ be the unique critical covering lattice of $B^n$,
and let $\Lambda_0$ be simple and $X(\Lambda_0)=\{S_1,\ldots,S_{2m}\}$ be semieutactic
with unique eutaxy coefficients $\upsilon_i$ be  such that $\sum_{i=1}^{2m}\upsilon_i Q_{S_i}=\idd$.
For each simplex $S_i$, denote by $\mathbf{x}_{ij}$, $j=1,\ldots,n+1$, its vertices and
by $\alpha_{ij}$ the corresponding barycentric coordinates of the circumcenter of $S_i$ (see \eqref{alpheqn}).
Let $K$ be a nearly spherical body
$(1-\eps)B^n\subseteq K\subseteq (1+\eps)B^n$, and
let $r_{ij}=1+\rho_{ij}$ be the values
of the radial distance function of $K$ evaluated at the directions $\mathbf{x}_{ij}$, $i=1,\ldots,2m$, $j=1,\ldots,n+1$.
There exists a covering lattice $\Lambda'$ of $K$ whose determinant is bounded as follows:
$$\frac{d(\Lambda')}{d(\Lambda_0)} \ge 1 + \sum_{i=1,j=1}^{n+1,2m}\upsilon_i\alpha_{ij}\rho_{ij} - \eps' \sum_{i=1,j=1}^{n+1,2m}|\rho_{ij}|\text,$$
where $\eps'$ depends on $\eps$ and becomes arbitrarily small as $\eps\to0$.
\end{thm}

\begin{proof}
We first prove the existence of a symmetric map $M$ and translation vectors $\mathbf{t}_i$, $i=1,\ldots,2m$ satisfying
$\trc M = \sum_{i=1,j=1}^{n+1,2m}\upsilon_i\alpha_{ij}\rho_{ij}$, and 
\begin{equation}
    \langle\mathbf{x}_{ij},M\mathbf{x}_{ij}+\mathbf{t}_i\rangle=\rho_{ij}\text{ for all }
    i=1,\ldots,2m\text{ and }j=1,\ldots,n+1\text.\label{treqn}
\end{equation}
Taking the sum $\sum_{j=1}^{n+1}\alpha_{ij}(\cdot)$ of both sides of \eqref{treqn},
we obtain
$$\sum_{j=1}^{n+1}\alpha_{ij}\langle \mathbf{x}_{ij},M\mathbf{x}_{ij}\rangle = \sum_{j=1}^{n+1}\alpha_{ij}\rho_{ij}\text.$$
Therefore, by the affine independence of the vertices of the simplex $S_i$, for fixed $i$ and $M$, a vector
$\mathbf{t}_i$ satisfying \eqref{treqn} for all $j=1,\ldots,n+1$ exists
if and only if $\sum_{j=1}^{n+1}\alpha_{ij}\rho_{ij} = \langle M,Q_{S_i}\rangle$. Let us denote
$\rho_{i}=\sum_{j=1}^{n+1}\alpha_{ij}\rho_{ij}$.

All that is left to do is to find a map $M$ such that $\langle M,Q_{S_i}\rangle =\rho_i$
for all $i=1,\ldots,2m$, and $\trc M = \sum_{i=1}^{2m}\upsilon_i\rho_i$. From the fact that
the eutaxy coefficients are unique (modulo the trivial degeneracy associated
with the fact that $Q_S=Q_{-S}$) and the fundamental theorem of linear algebra,
it is easy to see that such a map must exist regardless of the values of $\rho_{ij}$.
Moreover, the map $M$ and translations vectors $\mathbf{t}_i$ are unique and
depend linearly on $\rho_{ij}$.

\begin{figure}
    \begin{center}
	\includegraphics{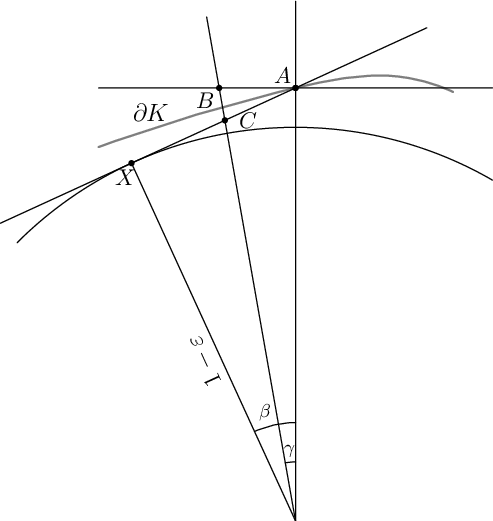}
	\caption{\label{tanfig}
	Illustration of the construction given in the proof
	of Theorem \ref{detthm} to bound the contraction factor needed
	to ensure that the original point $B$ when contracted to $C$
	lies inside the body $K$}
    \end{center}
\end{figure}

We now wish to find a contraction factor $1-\delta$ such that $||(1-\delta)\mathbf{y}_{ij}||\le r_K(\mathbf{y}_{ij}/||\mathbf{y}_{ij}||)$
for all $i,j$, where $\mathbf{y}_{ij}=(\idd+M)\mathbf{x}_{ij}+\mathbf{t}_i$. Therefore, for all $i,j$
we must have
$$\delta\ge\delta_{ij} = \frac{||\mathbf{y}_{ij}||-r_K(\mathbf{y}_{ij}/||\mathbf{y}_{ij}||)}{||\mathbf{y}_{ij}||}\text.$$
We wish to bound the values of $\delta_{ij}$ using only the values of the radial function evaluated at $\mathbf{x}_{ij}$ (not $\mathbf{y}_{ij}$)
and the fact that it is everywhere bounded between $1-\eps$ and $1+\eps$. We do this as illustrated in Figure \ref{tanfig}. In the
plane containing the origin $O$, the point $(1+\rho_{ij})\mathbf{x}_{ij}$ (denoted $A$ in the figure), and the point
$\mathbf{y}_{ij}$ (denoted $B$), draw the tangent $AX$ from $A$ to the circle of radius $1-\eps$ about the origin
in the direction toward $B$. Note that $B$ lies on the line through $A$ perpendicular to $OA$. Since $\rho_{ij}<\eps$,
the angle $\beta=\widehat{AOX}$ satisfies $\beta\le\cos^{-1}\dfrac{1-\eps}{1+\eps}\le2\sqrt{\eps}$. By convexity,
the segment $AX$ must lie in $K$. We mark the intersection of the tangent $AX$ and the ray $OB$ as $C$. Then
either $\delta_{ij}\le0$, or the boundary of $K$ intersects the ray $OB$ between $C$ and $B$. Since $\mathbf{y}_{ij}-\mathbf{x}_{ij}$
depends linearly on the values $\rho_{i'j'}$, the angle $\gamma=\widehat{AOB}$ satisfies $\gamma\le C\sum_{i'j'}|\rho_{i'j'}|$
for some constant $C$. By the law of sines we have
$$|BC|=\frac{|AB|\sin(\beta)}{\cos(\beta-\gamma)}\le\frac{(1+\epsilon)\gamma\beta}{1-\dfrac{1}{2}(\beta-\gamma)^2}\le(1+\eps)\eps'\sum_{i'j'}|\rho_{i'j'}|\text,$$
where $\eps'$ depends on $\eps$ and becomes arbitrarily small as $\eps\to 0$.
Therefore, if we let $\delta=\eps'\sum_{ij}|\rho_{ij}|$, then $\delta_{ij}\le\delta$ for all $i$ and $j$,
and for each simplex $S_i\in X(\Lambda_0)$, we now have that $(1-\delta)((\idd+M)S_i+\mathbf{t}_i)\subseteq K$.
Therefore, $\Lambda'=(1-\delta)(\idd+M)\Lambda_0$ is a covering
lattice for $K$.

The determinant of the lattice $\Lambda'$ is given by
$$\begin{aligned}\frac{d(\Lambda')}{d(\Lambda_0)} &= (1-\delta)^n \det(\idd+M)\\
&\ge\left(1+\sum_{i=1,j=1}^{2m,n+1}\upsilon_i\delta_{ij}\rho_{ij}-C\sum_{ij}|\rho_{ij}|^2\right)\left(1-\eps'\sum_{ij}|\rho_{ij}|\right)^n\\
&\ge1 + \sum_{i=1,j=1}^{n+1,2m}\upsilon_i\delta_{ij}\rho_{ij} - \eps'' \sum_{i=1,j=1}^{n+1,2m}|\rho_{ij}|\text,\end{aligned}$$
where the quadratic and higher order terms have been absorbed into the last term.
\end{proof}

\section{The case $n=3$}

We now turn to prove the main result, which is that the 3-dimensional
ball is locally pessimal. Given Theorem \ref{detthm}, the proof
proceeds much as the proof of Theorem 5 of Ref. \cite{kalluspack} does.
As in Ref. \cite{kalluspack}, we start with three lemmas, of which we will only prove
the first here, since it is the only one which varies significantly from
its analog in Ref. \cite{kalluspack}.

\begin{lem}\label{lemcl}Let 
$$c_l=P_l(1)+3P_l(\tfrac{4}{5})+P_l(\tfrac{3}{5})+4P_l(\tfrac{2}{5})+2P_l(\tfrac{1}{5})+P_l(0)\text,$$
where $P_l(t)$ is the Legendre polynomial of degree $l$.
Then $c_l=0$ if and only if $l=2$. Moreover, $|c_l-1|<C l^{-1/2}$ for some constant
$C$.\end{lem}
\begin{proof}
We introduce rescaled Legendre polynomials,
$Q_l(t) = 5^l l! P_l(t)$. From their recurrence relation---given by
$Q_{l+1}(t) = (2l+1) (5t) Q_l(t) - 25 l^2 Q_{l-1}(t)$---and the base cases---$Q_0(t)=1$ and $Q_1(t)=5t$---
it is clear that the values of $Q_l(t)$ at $t=k/5$ for $k=0,\ldots,5$ are integers.
We are interested in residues of these integers modulo $16$. If
$Q_l(k/5)\equiv Q_{l+1}(k/5)\equiv0\pmod{16}$ for some $k$ and $l$ then for all $l'\ge l$
we also have $Q_{l'}(k/5)\equiv0\pmod{16}$. This is in fact the case, as can be easily checked,
for $k=1,3,$ or $5$ and $l=6$. For $k=0,2,$ or $4$ it is easy to show by induction that
the residue of $Q_l(k/5)$ modulo $16$ depends only on $k$ and the residue of $l$ modulo $8$ and
takes the following values:
$$\begin{aligned}
Q_l(0) &\equiv& 1,0,7,0,9,0,7,0 \pmod{16}\\
Q_l(2/5) &\equiv& 4,8,12,8,4,8,12,8 \pmod{16}\\
Q_l(4/5) &\equiv& 3,12,5,4,11,12,5,4 \pmod{16}\\
\text{ resp.\ for } l&\equiv&0,1,2,3,4,5,6,7\pmod{8}\text.
\end{aligned}$$
Therefore, it is easy to verify that regardless of the residue of $l$ modulo $8$,
the quantity $5^l l! c_l = Q_l(1)+3Q_l(\tfrac{4}{5})+Q_l(\tfrac{3}{5})+4Q_l(\tfrac{2}{5})+2Q_l(\tfrac{1}{5})+Q_l(0)$
is an integer of non-zero residue modulo $16$ for $l\ge 6$ and therefore $c_l$ does not vanish.
The cases $l<6$ are easily checked by hand. The second part of the lemma
follows from the bound $|P_l(t)|<(\pi l\sqrt{1-t^2}/2)^{-1/2}$ \cite{bernstein}.
\end{proof}

Fixing some arbitrary pole $\mathbf{p}\in S^2$, we define a zonal measure (function) to
be a measure (function) on $S^2$ which
is invariant with respect to rotations that preserve $\mathbf{p}$. A convolution of a function $f$
with a zonal measure $\mu$ is given by $(\mu*f)(\mathbf{y}) = \int f(\mathbf{x}) d\mu(U_\mathbf{y}(\mathbf{x}))$,
where $U_\mathbf{y}$ is any rotation which takes $\mathbf{y}$ to $\mathbf{p}$.
Convolution of $f$ with a zonal measure acts as multiplier transformation on the 
harmonic expansion of $f$ \cite{convolutions}. That is, if $f(\mathbf{x}) = \sum_{l=0}^{\infty}f_l(\mathbf{x})$,
where $f_l(\mathbf{x})$ is a spherical harmonic of degree $l$,
then $(\mu*f)(\mathbf{x}) = \sum_{l=0}^{\infty}c_l f_l(\mathbf{x})$. 

Consider the 24 vertices of the Voronoi polytope of $A_3^*$ (the Archi\-medean truncated octahedron)
rotated in such a way that
one of them is at $\mathbf{p}$, and denote them as $\mathbf{x}_i$, $i=1,\ldots,24$.
There is a unique zonal measure $\mu$ such that for every continuous zonal function $f$
$$\int_{S^2}f(\mathbf{y})d\mu(\mathbf{y}) = \frac{1}{2} \sum_{i=1}^{24}f(\mathbf{x}_i)\text.$$
From the values $\langle\mathbf{p},\mathbf{x}_i\rangle$, $i=1,\ldots,24$, 
the multiplier coefficients associated with convolution with this measure
can be easily calculated (see Ref. \cite{convolutions}). It can be easily
shown that these coefficients vanish for odd $l$ and are equal to the
coefficients $c_l$ of Lemma \ref{lemcl} for even $l$.
The proof of the following two lemmas is identical to the proofs of Lemmas 4 and 2
of Ref. \cite{kalluspack} respectively.
We denote by $\sigma$ the Lebesgue measure on $S^2$ normalized such that $\sigma(S^2)=1$.

\begin{lem}\label{lemphi}Let $\mu$ be the zonal measure described above,
let $\Phi$ be the operator of convolution with
$\mu$, and let $Z$ be the space, equipped with the $L^1(\sigma)$ norm, of even
functions on $S^2$ for which $f_2=0$. Then $\Phi$ maps $Z$ to $Z$, and as an
operator $Z\to Z$ it is one-to-one, bounded, and has a bounded inverse.\end{lem}

\begin{lem}\label{lemr2}
Given $\eps>0$, there exists $\eps'>0$ such that if a convex body
$K$ satisfies $(1-\eps')B^3\subseteq K\subseteq (1+\eps')B^3$, then
$K$ has a linear image $K'=TK$ that satisfies $(1-\eps)B^3\subseteq K'\subseteq (1+\eps)B^3$
and whose radial function has mean $1$ and vanishing second spherical harmonic
component.\end{lem}

\begin{thm}\label{mainthm} There exists $\eps>0$ such that if a convex body $K$ is a non-ellipsoidal
origin-symmetric convex body and $(1-\eps)B^3\subseteq K\subseteq (1+\eps)B^3$,
then $\vartheta(K)<\vartheta(B^3)$. In other words, $B^3$ is relatively worst covering.\end{thm}

\begin{proof}
Given Lemma \ref{lemr2} and the fact that $\vartheta$ is invariant under linear transformations,
we may assume without loss of generality that $K$ is a non-spherical body whose radial function
has an expansion in spherical harmonics of the form
$$r_K(\mathbf{x}) = 1+\rho(\mathbf{x})=1+\sum_{l\text{ even},l\ge4}\rho_l(\mathbf{x})\text.$$
The volume of $K$ satisfies
\begin{equation}\label{voleqn}
\vll K = \frac{4\pi}{3}\int_{S^2} r_K^3(\mathbf{x})d\sigma \le \frac{4\pi}{3}+\eps''||\rho||_1\text,
\end{equation}
where $\eps''$ is arbitrarily small for arbitrarily small $\eps$. 

We consider all the rotations $U(K)$ of the body $K$ and the determinant
of the covering lattice obtained when the construction of Theorem \ref{detthm}
is applied to $U(K)$. Note that the determinant obtained depends only on
$\rho_{ij} = r_{U(K)}(\mathbf{x}_{ij}) -1 = \rho(U^{-1}(\mathbf{x}_{ij}))$,
where $\mathbf{x}_{ij}$ run over all 24 vertices of the three dimensional
permutohedron. Let us define
$\Delta_K=1-\tfrac{\vartheta(K)^{-1}}{\vartheta(B^n)^{-1}}$. Combining
\eqref{voleqn} with Theorem \ref{detthm} we get
\begin{equation}\label{del1}
\Delta_K\le \min_{U\in SO(3)}\left[-\frac{1}{8}\sum_{i=1,j=1}^{6,4}\rho_{ij}+\eps'\sum_{ij}|\rho_{ij}|+\eps''||\rho||_1\right]\text.
\end{equation}
We may pick a single point, say $\mathbf{x}_{11}$, and decompose $SO(3)$ into
subsets $\mathcal U_\mathbf{y}$ of all rotations such that $U^{-1}(\mathbf{x}_{11})=\mathbf{y}$.
In each subset $\mathcal U_\mathbf{y}$ the minimum on the right hand side of \eqref{del1}
is no larger than the average value over $\mathcal U_\mathbf{y}$ (with respect to the obvious
uniform measure). This averaging procedure transforms \eqref{del1} into
\begin{equation}\label{del2}
\Delta_K\le \min_{\mathbf{y}\in S^2}\left[-\frac{1}{8}\Phi[\rho](\mathbf{y})+\eps'\Phi[|\rho|](\mathbf{y})+\eps''||\rho||_1\right]\text,
\end{equation}
where $\Phi$ is the convolution operator in Lemma \ref{lemphi}. Since $\int \Phi[\rho] d\sigma=0$ and
$\Phi[|\rho|]$ is non-negative, we have that
$\min(-\tfrac{1}{8}\Phi[\rho]+\eps'\Phi[|\rho|])\le-\tfrac{1}{16}||\Phi[\rho]||_1+\eps'||\Phi[|\rho|]||_1$,
and so
$$\Delta_K\le-\frac{1}{16}||\Phi^{-1}||^{-1}\cdot||\rho||_1+\left(\eps'||\Phi||+\eps''\right)||\rho||_1\text.$$
Therefore, we conclude that there is a coefficient $c>0$ such that $\Delta_K<-c||\rho||_1$.\end{proof}

Theorem \ref{mainthm} and the fact that the disk is pessimal for covering in the plane motivate the following
conjecture:

\begin{cnj}For all origin-symmetric convex bodies $K\subseteq\mathbb{R}^3$ that are not ellipsoids, $\vartheta(K)<\vartheta(B^3)$.\end{cnj}

\bibliographystyle{amsplain}
\bibliography{cover}

\end{document}